\documentclass[a4paper , 12pt]{amsart}
\usepackage[top=30truemm,bottom=30truemm,left=25truemm,right=25truemm]{geometry}
\usepackage{amsmath}
\usepackage{amsthm}
\usepackage{amsfonts}
\usepackage{bm}
\newtheorem{thm}{Theorem}[section]
\newtheorem{lem}[thm]{Lemma}

\theoremstyle{definition}

\newtheorem*{rem}{Remark}

\DeclareMathOperator{\Dio}{Dio}

\DeclareMathOperator{\Card}{Card}
\makeatletter %\fB => \mathfrak{B}, \gb => \mathfrak{b}, \lB => \mathbb{B}, \cB => \mathcal{B}
\@tempcnta\z@
\loop\ifnum\@tempcnta<26
\advance\@tempcnta\@ne
\expandafter\edef\csname f\@Alph\@tempcnta\endcsname{\noexpand\mathfrak{\@Alph\@tempcnta}}
\expandafter\edef\csname l\@Alph\@tempcnta\endcsname{\noexpand\mathbb{\@Alph\@tempcnta}}
\expandafter\edef\csname c\@Alph\@tempcnta\endcsname{\noexpand\mathcal{\@Alph\@tempcnta}}
\repeat
\title{Mahler's classification and a certain class of $p$-adic numbers}
\author{Tomohiro Ooto}
\date{}
\begin{document}
\address{Graduate School of Pure and Applied Sciences, University of Tsukuba, Tennodai 1-1-1, Tsukuba, Ibaraki, 305-8571, Japan}
\email{ooto@math.tsukuba.ac.jp}

\begin{abstract}
In this paper, we study a relation between digits of $p$-adic numbers and Mahler's classification.
We show that an irrational $p$-adic number whose digits are automatic, primitive morphic, or Sturmian is an $S$-, $T$-, or $U_1$-number in the sense of Mahler's classification.
Furthermore, we give an algebraic independence criterion for $p$-adic numbers whose digits are Sturmian.
\end{abstract}

\maketitle

\section{Introduction}
Let $p$ be a prime.
We denote by $|\cdot |_p$ the absolute value of the $p$-adic number field $\lQ _p$ normalized to satisfy $|p|_p=1/p$.
We denote by $\lfloor x\rfloor $ the integer part and $\lceil x\rceil $ the upper integer part of a real number $x$.
We set $P:=\{0,1,\ldots ,p-1 \}$.

Let $\mathcal{A}$ be a finite set.
Let $\mathcal{A}^{*}, \mathcal{A}^{+}$, and $\mathcal{A}^{\lN }$ denote the set of finite words over $\mathcal{A}$, the set of nonempty finite words over $\mathcal{A}$, and the set of infinite words over $\mathcal{A}$, respectively.
We denote by $|W|$ the length of a finite word $W$ over $\mathcal{A}$.
For any integer $n\geq 0$, write $W^n=WW\ldots W$ ($n$ times repeated concatenation of the word $W$) and $\overline{W}=WW\ldots W\ldots $ (infinitely many times repeated concatenation of the word $W$).
Note that $W^0$ is equal to the empty word.
More generally, for any real number $w\geq 0$, write $W^{w}=W^{\lfloor w\rfloor}W'$, where $W'$ is the prefix of $W$ of length $\lceil (w-\lfloor w\rfloor )|W|\rceil $.
Let ${\bf a}=(a_n)_{n\geq 0}$ be a sequence over the set $\mathcal{A}$.
We identify ${\bf a}$ with the infinite word $a_0 a_1 \ldots a_n \ldots $.
An infinite word ${\bf a}$ over $\mathcal{A}$ is said to be {\itshape ultimately periodic} if there exist finite words $U\in \mathcal{A}^{*}$ and $V\in \mathcal{A}^{+}$ such that ${\bf a}=U\overline{V}$.

We recall the definition of automatic sequences, primitive morphic sequences, and Sturmian sequences.
Let $k\geq 2$ be an integer.
A {\itshape $k$-automaton} is a sextuplet
\begin{eqnarray*}
A=(Q, \Sigma _k, \delta , q_0, \Delta ,\tau ),
\end{eqnarray*} 
where $Q$ is a finite set, $\Sigma _k =\{ 0,1,\ldots ,k-1 \}$, $\delta :Q\times \Sigma _k\rightarrow Q$ is a map, $q_0 \in Q$, $\Delta $ is a finite set, and $\tau :Q\rightarrow \Delta $ is a map.
For an integer $n\geq 0$, we set $W_n:=w_m w_{m-1}\ldots w_0$, where $\sum _{i=0}^{m}w_i k^i$ is the $k$-ary expansion of $n$.
For $q\in Q$, we define recursively $\delta (q,W_n)$ by $\delta (q,W_n)=\delta (\delta (q,w_m w_{m-1}\ldots w_1), w_0)$.
A sequence ${\bf a}=(a_n)_{n\geq 0}$ is said to be {\itshape $k$-automatic} if there exists a $k$-automaton $A=(Q, \Sigma _k, \delta , q_0, \Delta ,\tau )$ such that $a_n=\tau (\delta (q_0,W_n))$ for all $n\geq 0$.
A sequence ${\bf a}=(a_n)_{n\geq 0}$ is said to be {\itshape automatic} if there exists an integer $k\geq 2$ such that ${\bf a}$ is $k$-automatic.

Let $\mathcal{A}$ and $\mathcal{B}$ be finite sets.
A map $\sigma :\mathcal{A}^{*}\rightarrow \mathcal{B}^{*}$ is said to be a {\itshape morphism} if $\sigma (UV)=\sigma (U)\sigma (V)$ for all $U, V \in \mathcal{A}^{*}$.
We define the {\itshape width} of $\sigma $ by $\max _{a\in \mathcal{A}}|\sigma (a)|$.
We say that $\sigma $ is {\itshape $k$-uniform} if there exists an integer $k\geq 1$ such that $|\sigma (a)|=k$ for all $a\in \mathcal{A}$.
In particular, we call a $1$-uniform morphism a {\itshape coding}.
A morphism $\sigma :\mathcal{A}^{*} \rightarrow \mathcal{A}^{*}$ is said to be {\itshape primitive} if there exists an integer $n\geq 1$ such that $a$ occurs in $\sigma ^n(b)$ for all $a, b \in \mathcal{A}$.
A morphism $\sigma :\mathcal{A}^{*}\rightarrow \mathcal{A}^{*}$ is said to be {\itshape prolongable} on $a\in \mathcal{A}$ if $\sigma (a)=aW$ where $W\in \mathcal{A}^{+}$, and $\sigma ^n(W)$ is not empty word for all $n\geq 1$.
We say that a sequence ${\bf a}=(a_n)_{n\geq 0}$ is {\itshape primitive morphic} if there exist finite sets $\mathcal{A},\mathcal{B}$, a primitive morphism $\sigma :\mathcal{A}^{*}\rightarrow \mathcal{A}^{*}$ which is prolongable on some $a \in\mathcal{A}$, and a cording $\tau :\mathcal{A}^{*}\rightarrow \mathcal{B}^{*}$ such that ${\bf a}=\lim_{n\rightarrow \infty }\tau (\sigma ^n(a))$.

Let $0<\theta <1$ be an irrational real number and $\rho $ be a real number.
For an integer $n\geq 1$, we put $s_{n, \theta ,\rho }:=\lfloor (n+1)\theta +\rho \rfloor - \lfloor n\theta +\rho \rfloor $ and $s_{n, \theta ,\rho } ':=\lceil (n+1)\theta +\rho \rceil -\lceil n\theta +\rho \rceil $.
Note that $s_{n, \theta ,\rho }, s_{n, \theta ,\rho } ' \in \{ 0,1 \}$.
We also put ${\bf s}_{\theta ,\rho }:=(s_{n, \theta ,\rho })_{n\geq 1}$ and ${\bf s}_{\theta ,\rho } ' :=(s_{n, \theta ,\rho } ')_{n\geq 1}$.
A sequence ${\bf a}=(a_n)_{n\geq 1}$ is called {\itshape Sturmian} if there exist an irrational real number $0<\theta <1$, a real number $\rho $, a finite set $\mathcal{A}$, and a coding $\tau :\{ 0,1\} ^{*}\rightarrow \mathcal{A}^{*}$ with $\tau (0)\not= \tau (1)$ such that $(a_n)_{n\geq 1}=(\tau (s_{n, \theta ,\rho }))_{n\geq 1}$ or $(\tau (s_{n, \theta ,\rho } '))_{n\geq 1}$.
Then we call $\theta $ (resp.\ $\rho $) the slope (resp.\ the intercept) of ${\bf a}$.

Applying so-called Subspace Theorem, Adamczewski and Bugeaud \cite{Adamczewski3} established a new transcendence criterion for $p$-adic numbers.

\begin{thm}\label{AB}
Let ${\bf a}=(a_n)_{n\geq 0}$ be a non-ultimately periodic sequence over $P$.
Set $\xi :=\sum_{n=0}^{\infty}a_n p^n \in \lQ _p$.
If the sequence ${\bf a}$ is automatic, primitive morphic, or Sturmian, then the $p$-adic number $\xi $ is transcendental.
\end{thm}

In this paper, we study $p$-adic numbers which satisfy the assumption of Theorem \ref{AB} in more detail.
For $\xi \in \lQ _p$ and an integer $n\geq 1$, we define $w_n(\xi )$ (resp.\ $w_n ^{*}(\xi )$) to be the supremum of the real number $w$ (resp.\ $w^{*}$) which satisfy
\begin{eqnarray*}
0<|P(\xi )|_p\leq H(P)^{-w-1}\quad (\mbox{resp.\ } 0<|\xi -\alpha |_p\leq H(\alpha )^{-w^{*}-1})
\end{eqnarray*}
for infinitely many integer polynomials $P(X)$ of degree at most $n$ (resp.\ algebraic numbers $\alpha \in \lQ _p$ of degree at most $n$).
Here, $H(P)$, which is called the {\itshape height} of $P(X)$, is defined by the maximum of the usual absolute values of the coefficients of $P(X)$ and $H(\alpha )$, which is called the {\itshape height} of $\alpha $, is defined by the height of the minimal polynomial of $\alpha $ over $\lZ$.
We set
\begin{eqnarray*}
w(\xi ):=\limsup _{n\rightarrow \infty }\frac{w_n(\xi )}{n}, \quad w^{*}(\xi ):=\limsup _{n \rightarrow \infty }\frac{w_n ^{*}(\xi )}{n}.
\end{eqnarray*}
A $p$-adic number $\xi $ is said to be an
\begin{gather*}
A \mbox{{\itshape -number} if } w(\xi )=0;\\
S \mbox{{\itshape -number} if } 0<w(\xi )<+\infty ;\\
T \mbox{{\itshape -number} if } w(\xi )=+\infty \mbox{ and } w_n(\xi )<+\infty \mbox{ for all } n;\\
U \mbox{{\itshape -number} if } w(\xi )=+\infty \mbox{ and } w_n(\xi )=+\infty \mbox{ for some } n.
\end{gather*}
Mahler \cite{Mahler1} first introduced the classification.
A $p$-adic number is algebraic if and only if it is an $A$-number.
Almost all $p$-adic numbers are $S$-numbers in the sense of Haar measure.
It is known that there exist uncountably many $T$-numbers.
Liouville numbers are $U$-numbers, for example $\sum _{n=1}^{\infty }p^{n!}$.
Replacing $w_n$ and $w$ with $w_n ^{*}$ and $w^{*}$, we define $A^{*}$-, $S^{*}$-, $T^{*}$-, and $U^{*}${\itshape -number} as above.
It is known that the two classification of $p$-adic numbers coincide.
Let $n\geq 1$ be an integer.
For a $U$-number (resp.\ a $U^{*}$-number) $\xi \in \lQ _p$, we say that $\xi $ is a $U_n${\itshape -number} (resp.\ a $U_n ^{*}${\itshape -number}) if $w_n(\xi )$ is infinite and $w_m(\xi )$ are finite (resp.\ $w_n ^{*}(\xi )$ is infinite and $w_m ^{*}(\xi )$ are finite) for all $1\leq m<n$.
The detail is found in \cite[Section 9.3]{Bugeaud1}.

We now state the main results.

\begin{thm}\label{app1}
Let ${\bf a}=(a_n)_{n\geq 0}$ be a non-ultimately periodic sequence over $P$.
Set $\xi :=\sum_{n=0}^{\infty}a_n p^n \in \lQ _p$.
If the sequence ${\bf a}$ is automatic, primitive morphic, or Sturmian with its slope whose continued fraction expansion has  bounded partial quotients, then the $p$-adic number $\xi $ is an $S$- or $T$-number.
Furthermore, if the sequence ${\bf a}$ is Strumian with its slope whose continued fraction expansion has unbounded partial quotients, then the $p$-adic number $\xi $ is a $U_1$-number.
\end{thm}

Theorem \ref{app1} is an extension of Theorem \ref{AB} and an analogue of Th\'eor\`emes 3.1, 4.2, and 5.1 in \cite{Adamczewski5}.

\begin{thm}\label{app2}
Let $\theta  >1$ be a real number whose continued fraction expansion has  bounded partial quotients, $\theta  '>1$ be a real number whose continued fraction expansion has unbounded partial quotients, and $\rho  ,\rho  '$ be real numbers.
Then the $p$-adic numbers
\begin{eqnarray*}
\sum _{n=1}^{\infty } p^{\lfloor n\theta +\rho  \rfloor },\quad \sum _{n=1}^{\infty } p^{\lfloor n\theta '+\rho  '\rfloor }
\end{eqnarray*}
are algebraically independent.
\end{thm}

Theorem \ref{app2} is an analogue of Corollaire 3.2 in \cite{Adamczewski5}.

This paper is organized as follows.
In Section \ref{‰ž—p}, we state Theorems \ref{main1} and \ref{main2}, and prove the main results assuming Theorems \ref{main1} and \ref{main2}.
We prepare some lemmas to prove Theorems \ref{main1} and \ref{main2} in Section \ref{€"õ}.
In Section \ref{Ø–¾}, we prove Theorems \ref{main1} and \ref{main2}.

\section{Extension of the main results}\label{‰ž—p}

Let ${\bf a}=(a_n)_{n\geq 0}$ be a sequence over a finite set $\mathcal{A}$.
The {\itshape $k$-kernel} of ${\bf a}=(a_n)_{n\geq 0}$ is the set of all sequences $(a_{k^i m+j})_{m\geq 0}$, where $i\geq 0$ and $0\leq j<k^i$.

Eilenberg \cite{Eilenberg} characterized $k$-automatic sequences.

\begin{lem}
Let $k\geq 2$ be an integer.
Then a sequence is $k$-automatic if and only if its $k$-kernel is finite.
\end{lem}

We say that the sequence ${\bf a}=(a_n)_{n\geq 0}$ is {\itshape $k$-uniform morphic} if there exist finite sets $\mathcal{A},\mathcal{B}$, a $k$-uniform morphism $\sigma :\mathcal{A}^{*}\rightarrow \mathcal{A}^{*}$ which is prolongable on some $a \in\mathcal{A}$, and a coding $\tau :\mathcal{A}^{*}\rightarrow \mathcal{B}^{*}$ such that ${\bf a}=\lim_{n\rightarrow \infty }\tau (\sigma ^n(a))$.
Then we call $\mathcal{A}$ the {\itshape initial alphabet} associated with ${\bf a}$.

Cobham \cite{Cobham} showed another characterization of $k$-automatic sequences using $k$-uniform morphic sequences.

\begin{lem}
Let $k\geq 2$ be an integer.
Then a sequence is $k$-automatic if and only if it is $k$-uniform morphic.
\end{lem}

The {\itshape complexity function} of the sequence ${\bf a}$ is given by
\begin{eqnarray*}
p({\bf a}, n):= \Card \{a_i a_{i+1}\ldots a_{i+n-1}\mid i\geq 0 \} ,\quad \mbox{for } n\geq 1.
\end{eqnarray*}

Let $\rho $ be a real number.
We say that ${\bf a}$ satisfies {\itshape Condition $(*)_{\rho }$} if there exist sequences of finite words $(U_n)_{n\geq 1}$, $(V_n)_{n\geq 1}$ and a sequence of nonnegative real numbers $(w_n)_{n\geq 1}$ such that
\begin{description}
 \item[(i)] the word $U_n V_n ^{w_n}$ is the prefix of ${\bf a}$ for all $n\geq 1$,
 \item[(ii)] $|U_n V_n ^{w_n}|/|U_n V_n|\geq \rho $ for all $n\geq 1$,
 \item[(iii)] the sequence $(|V_n ^{w_n}|)_{n\geq 1}$ is strictly increasing.
\end{description}
The {\itshape Diophantine exponent} of ${\bf a}$, first introduced in \cite{Adamczewski2}, is defined to be the supremum of a real number $\rho $ such that ${\bf a}$ satisfy Condition $(*)_{\rho }$.
We denote by $\Dio ({\bf a})$ the Diophantine exponent of ${\bf a}$.
It is immediate that
\begin{eqnarray*}
1\leq \Dio ({\bf a})\leq +\infty .
\end{eqnarray*}

We recall known results about Diophantine exponents and complexity function for automatic sequences, primitive morphic sequences, and Strumian sequences.

Adamczewski and Cassaigne \cite{Adamczewski1} estimated the Diophantine exponent of $k$-automatic sequences.

\begin{lem}\label{auto1}
Let $k\geq 2$ be an integer.
Let ${\bf a}$ be a non-ultimately periodic and $k$-automatic sequence.
Let $m$ be a cardinality of the $k$-kernel of ${\bf a}$.
Then we have
\begin{eqnarray*}
\Dio ({\bf a})<k^m.
\end{eqnarray*}
\end{lem}

Moss\'e's result \cite{Mosse} implies the following lemma.

\begin{lem}\label{pri1}
Let ${\bf a}$ be a non-ultimately periodic and primitive morphic sequence.
Then the Diophantine exponent of ${\bf a}$ is finite.
\end{lem}

Adamczewski and Bugeaud \cite{Adamczewski5} established a relation between Strumian sequences and Diophantine exponents.

\begin{lem}\label{Strumian1}
Let ${\bf a}$ be a Strumian sequence with slope $\theta $.
Then the continued fraction expansion of $\theta $ has bounded partial quotients if and only if the Diophantine exponent of ${\bf a}$ is finite.
\end{lem}

It is known that automatic sequences, primitive morphic sequences, and Sturmian sequences have low complexity.

\begin{lem}\label{auto2}
Let $k\geq 2$ be an integer and ${\bf a}$ be a $k$-automatic sequence.
Let $d$ be a cardinality of the internal alphabet associated with ${\bf a}$.
Then we have for all $n\geq 1$
\begin{eqnarray*}
p({\bf a}, n) \leq k d^2 n.
\end{eqnarray*}
\end{lem}

\begin{proof}
See \cite[Theorem 10.3.1]{Allouche} or \cite{Cobham}.
\end{proof}

\begin{lem}\label{pri2}
Let ${\bf a}$ be a primitive morphic sequence over a finite set of cardinality of $b\geq 2$.
Let $v$ be the width of a primitive morphism $\sigma $ which generates the sequence ${\bf a}$.
Then we have for all $n\geq 1$
\begin{eqnarray*}
p({\bf a},n)\leq 2 v^{4 b-2} b^3 n.
\end{eqnarray*}
\end{lem}

\begin{proof}
See \cite[Theorem 10.4.12]{Allouche}.
\end{proof}

\begin{lem}\label{Strumian2}
Let ${\bf a}$ be a Sturmian sequence.
Then we have for all $n\geq 1$
\begin{eqnarray*}
p({\bf a}, n)=n+1.
\end{eqnarray*}
\end{lem}

\begin{proof}
See \cite[Theorem 10.5.8]{Allouche}.
\end{proof}

\begin{thm}\label{main1}
Let ${\bf a}=(a_n)_{n\geq 0}$ be a non-ultimately periodic sequence over $P$.
Set $\xi :=\sum_{n=0}^{\infty}a_n p^n \in \lQ _p$.
Assume that there exist integers $n_0\geq 1$ and $\kappa \geq 2$ such that for all $n\geq n_0$,
\begin{eqnarray*}
p({\bf a}, n) \leq \kappa n.
\end{eqnarray*}
Then the $p$-adic number $\xi $ is an $S$-, $T$-, or $U_1$-number.
\end{thm}

Theorem \ref{main1} is an analogue of Th\'eor\`eme 1.1 in \cite{Adamczewski5}.
There is a real continued fraction analogue of Theorem \ref{main1} in \cite[Theorem 3.2]{Bugeaud2}.

\begin{thm}\label{main2}
Let ${\bf a}=(a_n)_{n\geq 0}$ be a non-ultimately periodic sequence over $P$.
Set $\xi :=\sum_{n=0}^{\infty}a_n p^n \in \lQ _p$.
Assume that there exist integers $n_0\geq 1$ and $\kappa \geq 2$ such that for all $n\geq n_0$,
\begin{eqnarray*}
p({\bf a}, n) \leq \kappa n.
\end{eqnarray*}
Then the Diophantine exponent of ${\bf a}$ is finite if and only if $\xi $ is not a $U_1$-number.
Furthermore, if the Diophantine exponent of ${\bf a}$ is finite, then we have
\begin{eqnarray}\label{last}
w_1(\xi )\leq 8(\kappa +1)^2(2\kappa +1)\Dio ({\bf a})-1.
\end{eqnarray}
\end{thm}

There are various versions of Theorem \ref{main2}: $b$-ary expansion for real numbers \cite{Adamczewski5}, continued fraction expansion for real numbers \cite{Bugeaud2}, formal power series over a finite field, and its continued fraction expansion \cite{Ooto}.

\begin{proof}[Proof of Theorem \ref{app1} assuming Theorems \ref{main1} and \ref{main2}]
Since the sequence ${\bf a}$ is automatic, primitive morphic, or Strumian, $\xi $ is an $S$-, $T$-, or $U_1$-number by Lemmas \ref{auto2}, \ref{pri2}, \ref{Strumian2} and Theorem \ref{main1}.
It follows from Lemmas \ref{auto1}, \ref{pri1}, \ref{Strumian1} and Theorem \ref{main2} that $\xi $ is a $U_1$-number if ${\bf a}$ is Strumian with its slope whose continued fraction expansion has unbounded partial quotients, and $\xi $ is an $S$- or $T$-number otherwise.
\end{proof}

Let $\theta $ and $\rho $ be real numbers.
For an integer $n\geq 1$, we put
\begin{gather*}
t_n :=
\begin{cases}
1 &  \mbox{if } n=\lfloor k\theta +\rho \rfloor \mbox{ for some integer } k, \\
0 &  \mbox{otherwise},
\end{cases}
\\
t_n ' :=
\begin{cases}
1 &  \mbox{if } n=\lceil k\theta +\rho \rceil \mbox{ for some integer } k, \\
0 &  \mbox{otherwise}.
\end{cases}
\end{gather*}
We also put ${\bf t}_{\theta ,\rho }:=(t_n)_{n\geq 1}$ and ${\bf t}_{\theta ,\rho } ' :=(t_n ')_{n\geq 1}$.
The lemma below is well-known result.

\begin{lem}\label{St}
Let $\theta >1$ be an irrational real number and $\rho $ be a real number.
Then we have ${\bf t}_{\theta ,\rho }={\bf s}_{1/\theta ,-(\rho +1)/\theta } '$ and ${\bf t}_{\theta ,\rho } '={\bf s}_{1/\theta ,-(\rho +1)/\theta }$.
\end{lem}

\begin{lem}[Mahler \cite{Mahler1}]\label{Ma}
Let $\xi ,\eta $ be $p$-adic numbers.
If $\xi $ and $\eta $ are algebraically dependent, then $\xi $ and $\eta $ are in the same class.
\end{lem}

\begin{proof}[Proof of Theorem \ref{app2} assuming Theorems \ref{main1} and \ref{main2}]
Set $\xi :=\sum _{n=1}^{\infty } p^{\lfloor n\theta +\rho  \rfloor }$ and \\ $\eta:=\sum _{n=1}^{\infty } p^{\lfloor n\theta '+\rho  '\rfloor }$.
By the definition, the digits of $\xi $ and $\eta $ are ${\bf t}_{\theta ,\rho }$ and ${\bf t}_{\theta ',\rho '}$, respectively.
It follows from Lemma \ref{St} that ${\bf t}_{\theta ,\rho }$ (resp.\ ${\bf t}_{\theta ',\rho '}$) is Strumian with its slope whose continued fraction expansion has bounded (resp.\ unbounded) partial quotients.
Therefore, $\xi $ is an $S$- or $T$-number and $\eta $ is a $U_1$-number by Theorem \ref{app1}.
Hence, we see that $\xi $ and $\eta $ are algebraically independent by Lemma \ref{Ma}.
\end{proof}

\section{Preliminaries}\label{€"õ}

We recall several facts about the exponents $w_n$ and $w_n ^{*}$.

\begin{thm}\label{"äŠr}
Let $n\geq 1$ be an integer and $\xi $ be in $\lQ _p$.
Then we have
\begin{eqnarray*}
w_n ^{*}(\xi )\leq w_n(\xi )\leq w_n ^{*}(\xi )+n-1.
\end{eqnarray*} 
\end{thm}

\begin{proof}
See \cite{Morrison}.
\end{proof}

\begin{thm}\label{Mahler‰ºŒÀ}
Let $n\geq 1$ be an integer and $\xi \in \lQ _p$ be not algebraic of degree at most $n$.
Then we have
\begin{align*}
w_n (\xi ) \geq n,\quad w_n ^{*}(\xi ) & \geq \frac{n+1}{2}.
\end{align*}
Furthermore, if $n=2$, then $w_2 ^{*}(\xi )\geq 2$.
\end{thm}

\begin{proof}
See \cite{Mahler1, Morrison}.
\end{proof}

We recall Liouville inequality, that is, a non trivial lower bound of differences of two algebraic numbers.

\begin{lem}\label{Liouville inequality}
Let $\alpha ,\beta \in \lQ _p$ be distinct algebraic numbers of degree $m, n$, respectively.
Then we have
\begin{eqnarray*}
|\alpha -\beta |_p \geq \frac{(m+1)^{-n}(n+1)^{-m}}{H(\alpha )^n H(\beta )^m}.
\end{eqnarray*}
\end{lem}

\begin{proof}
See \cite[Lemma 2.5]{Pejkovic}.
\end{proof}

Applying Lemma \ref{Liouville inequality}, we give an estimate for the value of $w_1$.

\begin{lem}\label{–³—""x"äŠr}
Let $\xi $ be in $\lQ _p$ and $c_0,c_1,c_2,\theta ,\rho ,\delta $ be positive numbers.
Let $(\beta _j)_{j\geq 1}$ be a sequence of positive integers with $\beta _j<\beta _{j+1}\leq c_0 \beta _j ^\theta $ for all $j\geq 1$.
Assume that there exists a sequence of distinct terms $(\alpha _j)_{j\geq 1}$ with $\alpha _j \in \lQ$ such that for all $j\geq 1$
\begin{gather*}
\frac{c_1}{\beta _j ^{1+\rho }}\leq \left| \xi -\alpha _j \right| _p \leq \frac{c_2}{\beta _j ^{1+\delta }}, \\
H(\alpha _j)\leq c_3 \beta _j.
\end{gather*}
Then we have
\begin{eqnarray*}
\delta \leq w_1(\xi )\leq (1+\rho )\frac{\theta }{\delta }-1.
\end{eqnarray*}
\end{lem}

\begin{rem}
There are several versions of Lemma \ref{–³—""x"äŠr} as in \cite{Adamczewski4, Adamczewski5, Alladi1, Amou1, Bugeaud2, Bugeaud3, Danilov1, Firicel, Ooto, Voloch1}.
\end{rem}

\begin{proof}
Let $\alpha $ be a rational number with sufficiently large height.
We define the integer $j_0\geq 1$ by $\beta _{j_0}\leq c_0 (4c_2 c_3 H(\alpha ))^{\theta /\delta }<\beta _{j_0+1}$.
Firstly, we consider the case $\alpha =\alpha _{j_0}$.
By the assumption, we obtain
\begin{eqnarray*}
|\xi -\alpha |_p\geq c_1 \beta _{j_0} ^{-1-\rho }\geq c_0 ^{-1-\rho }c_1 (4 c_2 c_3)^{-(1+\rho )\theta /\delta }H(\alpha )^{-(1+\rho )\theta /\delta }.
\end{eqnarray*}
Next, we consider the other case.
Then, by the assumption, we have
\begin{eqnarray*}
H(\alpha )<(4 c_2 c_3 )^{-1} (c_0 ^{-1} \beta _{j_0+1})^{\delta /\theta }\leq (4 c_2 c_3 )^{-1} \beta _{j_0} ^{\delta }.
\end{eqnarray*}
Therefore, we obtain
\begin{eqnarray*}
|\alpha -\alpha _{j_0}|_p \geq (4 H(\alpha )H(\alpha _{j_0}))^{-1} >c_2 \beta _{j_0} ^{-1-\delta }
\end{eqnarray*}
by Lemma \ref{Liouville inequality}.
Hence, it follows that
\begin{eqnarray*}
|\xi -\alpha |_p & = & |\alpha -\alpha _{j_0}|_p\geq (4 H(\alpha )H(\alpha _{j_0}))^{-1}\\
& \geq & 4^{-1-\theta /\delta }c_0 ^{-1} c_2 ^{-\theta /\delta } c_3 ^{-1-\theta /\delta }H(\alpha )^{-1-\theta /\delta }.
\end{eqnarray*}
By Theorem \ref{"äŠr}, we have $w_1(\xi )=w_1 ^{*}(\xi )$.
Thus, we obtain
\begin{eqnarray*}
\delta \leq w_1(\xi )\leq \max \left( (1+\rho )\frac{\theta }{\delta }-1, \frac{\theta }{\delta } \right) =(1+\rho )\frac{\theta }{\delta }-1.
\end{eqnarray*}
\end{proof}

We denote by $M_{\lQ }$ the set of all prime numbers and $\infty $.
We denote by $|\cdot |_{\infty }$ the usual absolute value in $\lQ $. 
For ${\bf x}=(x_1,\ldots ,x_n)\in \lQ ^n$ and $v \in M_{\lQ }$, we define the {\itshape norm} and the {\itshape height} of ${\bf x}$ by $|{\bf x}| _v=\max _{1\leq i\leq n}|x_i|_v$ and $H({\bf x})=\prod _{v\in M_{\lQ }}^{} |{\bf x}|_v$.

The proof of Theorem \ref{main1} mainly depends on the following theorem which is so-called Quantitative Subspace Theorem and consequence of Corollary 3.2 in \cite{Evertse}.

\begin{thm}\label{subsp}
Let $\alpha \in \lQ _p$ be an algebraic number of degree $d$ and $0<\varepsilon <1$.
Define linear forms
\begin{gather*}
L_{1\infty }(X,Y,Z)=X,\quad L_{2\infty }(X,Y,Z)=Y,\quad L_{3\infty }(X,Y,Z)=Z, \\
L_{1 p}(X,Y,Z)=X,\quad L_{2 p}(X,Y,Z)=Y,\quad L_{3 p}(X,Y,Z)=\alpha X-\alpha Y-Z.
\end{gather*}
Then all integer solutions ${\bf x}=(x_1,x_2,x_3)$ of
\begin{eqnarray*}
\prod _{v\in \{ p,\infty \}}^{}\prod _{i=1}^{3}|L_{i v}({\bf x})|_v\leq |{\bf x}|_{\infty } ^{-\varepsilon }
\end{eqnarray*}
with
\begin{eqnarray*}
H({\bf x})\geq \max \left( \left( \sqrt{d+1}H(\alpha )\right) ^{1/12 d}, 27^{1/\varepsilon }\right)
\end{eqnarray*}
lie in the union of at most
\begin{eqnarray*}
2^{16} 3^{39} 5^{10} \varepsilon ^{-9} \log (3\varepsilon ^{-1}d) \log (\varepsilon ^{-1}\log 3 d)
\end{eqnarray*}
proper linear subspaces of $\lQ ^3$.
\end{thm}

Consider a vector hyperplane of $\lQ ^n$
\begin{eqnarray*}
\mathcal{H}=\{ (x_1,\ldots ,x_n) \in \lQ ^n \mid y_1 x_1+\cdots +y_n x_n=0 \} ,
\end{eqnarray*}
where ${\bf y}=(y_1,\ldots ,y_n) \in \lZ ^n\setminus \{ {\bf 0} \} , \gcd (y_1,\ldots ,y_n)=1$.
The {\itshape height} of $\mathcal{H}$, denoted by $H(\mathcal{H})$, is defined to be $|{\bf y}|_{\infty }$.

The lemma below is easily seen.

\begin{lem}\label{hyper}
Let $m, n$ be integers with $1\leq m<n$ and ${\bf x}_1,\ldots ,{\bf x}_m \in \lZ ^n$ be linearly independent vectors such that $|{\bf x}_1|_{\infty }\leq \ldots \leq |{\bf x}_m|_{\infty }$.
Then there exists a vector hyperplane $\mathcal{H}$ of $\lQ ^n$ such that ${\bf x}_1,\ldots ,{\bf x}_m \in \mathcal{H}$ and
\begin{eqnarray*}
H(\mathcal{H})\leq m! |{\bf x}_m|_{\infty }^m.
\end{eqnarray*}
\end{lem}

\begin{lem}\label{''³ãŒÀ}
Let $U \in P^{*}, V\in P^{+}$, and $r,s$ be lengths of the words $U,V$, respectively.
Put $(a_n)_{n\geq 0}:=U\overline{V}$ and $\alpha :=\sum _{n=0}^{\infty }a_n p^n \in \lQ _p$.
Then we have $H(\alpha )\leq p^{r+s}$.
\end{lem}

\begin{proof}
A straightforward computation shows that
\begin{eqnarray*}
\alpha & = & \sum _{n=0}^{r-1}a_n p^n+\left( \sum _{m=0}^{s-1}a_{m+r}p^{m+r}\right) \left( \sum _{k=0}^{\infty }p^{ks}\right) \\
& = & \frac{(p^s-1)\sum _{n=0}^{r-1}a_n p^n-\sum _{m=0}^{s-1}a_{m+r}p^{m+r}}{p^s-1}.
\end{eqnarray*}
Therefore, we have 
\begin{eqnarray*}
H(\alpha )\leq \max \left( p^s-1, (p^s-1)\sum _{n=0}^{r-1}a_n p^n, \sum _{m=0}^{s-1}a_{m+r}p^{m+r} \right) \leq p^{r+s}.
\end{eqnarray*}
\end{proof}

In order to prove Theorem \ref{main2}, we show the following lemma.

\begin{lem}\label{•Ð'¤}
Let ${\bf a}=(a_n)_{n\geq 0}$ be a non-ultimately periodic sequence over $P$.
Set $\xi :=\sum_{n=0}^{\infty}a_n p^n \in \lQ _p$.
Then we have
\begin{eqnarray}\label{‰º}
w_1(\xi )\geq \max (1, \Dio ({\bf a})-1).
\end{eqnarray}
\end{lem}

\begin{proof}
Since $\xi $ is irrational, we have $w_1(\xi )\geq 1$ by Theorem \ref{Mahler‰ºŒÀ}.
Without loss of generality, we may assume that $\Dio ({\bf a})>1$.
Take a real number $\delta $ such that $1<\delta <\Dio ({\bf a})$.
For $n\geq 1$, there exist finite words $U_n,V_n $ and a positive rational number $w_n $ such that $U_n V_n ^{w_n}$ are the prefix of ${\bf a}$, the sequence $(|V_n ^{w_n}|)_{n\geq 1}$ is strictly increasing, and $|U_n V_n ^{w_n}|\geq \delta |U_n V_n|$.
For $n\geq 1$, we set rational number
\begin{eqnarray*}
\alpha _n:=\sum _{i=0}^{\infty }b_i ^{(n)} p^i
\end{eqnarray*}
where $(b_i ^{(n)})_{i\geq 0}$ is the infinite word $U_{n}\overline{V_n}$.
Since $\xi $ and $\alpha _n$ have the same first $|U_n V_n ^{w_n}|$-th digits, we obtain
\begin{eqnarray*}
\left| \xi -\alpha _n \right| \leq p^{-\delta |U_n V_n|}\leq H(\alpha _n)^{-\delta }
\end{eqnarray*}
by Lemma \ref{''³ãŒÀ}.
Hence, we have (\ref{‰º}).
\end{proof}

The following lemma is a slight improvement of a part of Lemma 9.1 in \cite{Bugeaud2}.

\begin{lem}\label{'g'ݍ‡'킹}
Let ${\bf a}=(a_n)_{n\geq 0}$ be a sequence on a finite set $\mathcal{A}$.
Assume that there exist integers $\kappa \geq 2$ and $n_0\geq 0$ such that for all $n\geq n_0$,
\begin{eqnarray*}
p({\bf a},n)\leq \kappa n.
\end{eqnarray*}
Then, for each $n\geq n_0$, there exist finite words $U_n, V_n$ over $\mathcal{A}$ and a positive rational number $w_n$ such that the following hold:
\begin{description}
\item[(i)] $U_n V_n ^{w_n}$ is a prefix of ${\bf a}$,
\item[(ii)] $|U_n|\leq 2\kappa |V_n|$,
\item[(iii)] $n/2 \leq |V_n|\leq \kappa n$,
\item[(iv)] if $U_n$ is not an empty word, then the last letter of $U_n$ and $V_n$ are different,
\item[(v)] $|U_n V_n ^{w_n}|/|U_n V_n|\geq 1+1/(4\kappa +2)$,
\item[(vi)] $|U_n V_n|\leq (\kappa +1)n-1$,
\end{description}
\end{lem}

\begin{proof}
For $n\geq 1$, we denote by $A(n)$ the prefix of ${\bf a}$ of length $n$.
By Pigeonhole principle, for each $n\geq n_0$, there exists a finite word $W_n$ of length $n$ such that the word appears to $A((\kappa +1)n)$ at least twice.
Therefore, for each $n\geq n_0$, there exist finite words $B_n, D_n, E_n \in \mathcal{A}^{*}$ and $C_n \in \mathcal{A}^{+}$ such that
\begin{eqnarray*}
A((\kappa +1)n)=B_n W_n D_n E_n =B_n C_n W_n E_n.
\end{eqnarray*}
We take these words in such way that if $B_n$ is not empty, then the last letter of $B_n$ is different from that of $C_n$.

We first consider the case of $|C_n|\geq |W_n|$.
Then, there exists $F_n \in \mathcal{A}^{*}$ such that
\begin{eqnarray*}
A((\kappa +1)n)=B_n W_n F_n W_n E_n.
\end{eqnarray*}
Put $U_n :=B_n, V_n:=W_n F_n$, and $w_n:=|W_n F_n W_n|/|W_n F_n|$.
Since $U_n V_n ^{w_n}=B_n W_n F_n W_n$, the word $U_n V_n ^{w_n}$ is a prefix of ${\bf a}$.
It is obvious that $|U_n|\leq (\kappa -1)|V_n|$ and $n\leq |V_n|\leq \kappa n$.
By the definition, we have (iv) and (vi).
Furthermore, we see that
\begin{gather*}
\frac{|U_n V_n ^{w_n}|}{|U_n V_n|}=1+\frac{n}{|U_n V_n|}\geq 1+\frac{1}{\kappa }.
\end{gather*}

We next consider the case of $|C_n|< |W_n|$.
Since the two occurrences of $W_n$ do overlap, there exists a rational number $d_n>1$ such that $W_n=C_n ^{d_n}$.
Put $U_n:=B_n, V_n:=C_n ^{\lceil d_n /2 \rceil}$, and $w_n:=(d_n +1)/\lceil d_n/2 \rceil$.
Obviously, we have (i) and (iv).
Since $ \lceil d_n/2 \rceil \leq d_n$ and $d_n |C_n|\leq 2 \lceil d_n/2 \rceil |C_n|$, we get $n/2\leq |V_n|\leq n$.
Using (iii) and $|U_n|\leq \kappa n-1$, we see (ii) and (vi).
It is immediate that $w_n \geq 3/2$.
Hence, we obtain
\begin{eqnarray*}
\frac{|U_n V_n ^{w_n}|}{|U_n V_n|} & = & 1+\frac{ \lceil (w_n-1)|V_n| \rceil}{|U_n V_n|} \geq 1+ \frac{w_n-1}{|U_n|/|V_n| +1} \\
& \geq & 1+\frac{1/2}{2\kappa +1}= 1+\frac{1}{4\kappa +2}.
\end{eqnarray*}
\end{proof}

\section{Proof of Theorems \ref{main1} and \ref{main2}}\label{Ø–¾}

\begin{proof}[Proof of Theorem \ref{main1}]
By Theorem 1B in \cite{Adamczewski3}, $\xi $ is transcendental, that is, $\xi $ is not an $A$-number.
Therefore, it is sufficient to prove that if $\xi $ is not a $U_1$-number, then $\xi $ is not a $U$-number.
For $n\geq n_0$, we take finite words $U_n , V_n$ over $P$ and positive rational numbers $w_n$ satisfying Lemma \ref{'g'ݍ‡'킹} (i)-(vi).
We define a positive integer sequence $(n_k)_{k\geq 0}$ by $n_{k+1}=4(\kappa +1)n_k$ for $k\geq 0$.
We set $r_k:=|U_{n_k}|, s_k:=|V_{n_k}|$, and $t_k:=|U_{n_k} V_{n_k}|$ for $k\geq 0$.
Then a straightforward computation shows that $2 t_k\leq t_{k+1}\leq c t_k, r_k\leq 2 \kappa s_k$ for $k\geq 0$, and $(s_k)_{k\geq 0}$ is strictly increasing, where $c=8(\kappa +1)^2$.
For $k\geq 0$, there exists an integer $p_k$ such that
\begin{eqnarray*}
\frac{p_k}{p^{s_k}-1}=\sum _{i=0}^{\infty } b_i ^{(k)} p^i
\end{eqnarray*}
where $(b_i ^{(k)})_{i\geq 0}$ is the infinite word $U_{n_k}\overline{V_{n_k}}$.
Since $\xi $ and $p_k/(p^{s_k}-1)$ have the same first $|U_{n_k}V_{n_k} ^{w_{n_k}}|$-th digits, we obtain
\begin{eqnarray*}
\left| \xi - \frac{p_k}{p^{s_k}-1} \right| _p\leq p^{-w t_k},
\end{eqnarray*}
where $w=1+1/(4\kappa +2)$.
Since the sequence $(s_k)_{k\geq 1}$ is strictly increasing, we may assume that $t_0\geq 3$.

Let $\alpha \in \lQ _p$ be an algebraic number of degree $d\geq 2$ with $H(\alpha )\geq \max (d+1,p^{s_0},27^{4\kappa +2})$.
We define an integer $j\geq 1$ by $p^{s_{j-1}}\leq H(\alpha )<p^{s_j}$ and a real number $\chi $ by $|\xi -\alpha |_p =H(\alpha )^{-\chi }$.
Without loss of generality, we may assume that $\chi >0$.
Put $M:=\max \{ m \in \lZ \mid p^{wc^{m-1}t_j}<H(\alpha )^{\chi } \} $.
In what follows, we estimate an upper bound of $M$.
Therefore, we may assume that $M\geq 1$.
Then we obtain $p^{w t_{j+h}}\leq p^{w c^{M-1}t_j}$ for all $0\leq h\leq M-1$.
Therefore, we have
\begin{eqnarray*}
|p^{s_{j+h}}\alpha -\alpha -p_{j+h}|_p & = & \left| \alpha - \frac{p_{j+h}}{p^{s_{j+h}}-1} \right| _p \\
& \leq & \max \left( \left| \xi - \frac{p_{j+h}}{p^{s_{j+h}}-1} \right| _p, |\xi - \alpha |_p\right) \leq p^{-w t_{j+h}}
\end{eqnarray*}
for $0\leq h\leq M-1$.
We define linear forms by
\begin{gather*}
L_{1\infty }(X,Y,Z)=X,\quad L_{2\infty }(X,Y,Z)=Y,\quad L_{3\infty }(X,Y,Z)=Z, \\
L_{1 p}(X,Y,Z)=X,\quad L_{2 p}(X,Y,Z)=Y,\quad L_{3 p}(X,Y,Z)=\alpha X-\alpha Y-Z,
\end{gather*}
and put ${\bf x}_h:=(p^{s_{j+h}},1,p_{j+h})$ for $0\leq h\leq M-1$.
By the proof of Lemma \ref{''³ãŒÀ}, we obtain
\begin{eqnarray*}
\prod _{v\in \{p,\infty \}}^{}\prod _{i=1}^{3} |L_{i v}({\bf x}_h)|_v\leq |{\bf x}_h|_{\infty } ^{-1/(4\kappa +2)}
\end{eqnarray*}
for all $0\leq h\leq M-1$.
We also have
\begin{eqnarray*}
H({\bf x}_h)=|{\bf x}_h|_{\infty }\geq p^{s_{j+h}}\geq H(\alpha )\geq \max \left( \left( \sqrt{d+1}H(\alpha ) \right) ^{1/12 d}, 27^{4\kappa +2}\right)
\end{eqnarray*}
for all $0\leq h\leq M-1$.
Hence, by Theorem \ref{subsp}, for all $0\leq h\leq M-1$, we obtain ${\bf x}_h$ in the union of $N$ proper linear subspaces of $\lQ ^3$, where $N=\lfloor 2^{25} 3^{39} 5^{10}(2\kappa +1)^9 \log (6(2\kappa +1)d) \log (2(2\kappa +1) \log 3 d) \rfloor $.

Assume that one of these linear subspaces of $\lQ ^3$ contains $L+1$ points of the set $\{{\bf x}_h \mid 0\leq h\leq M-1\}$, where $L=\lceil \log _2 ((2\kappa +1)(4d+6 +\log _p (2^{2d+1}(d+1)))) \rceil$.
It follows that there exist $(x,y,z)\in \lZ ^3 \setminus \{ {\bf 0} \}$ such that
\begin{eqnarray*}
x p^{s_{j+i_k}}+y+z p_{j+i_k}=0,\quad (0\leq k\leq L),
\end{eqnarray*}
where $0\leq i_0< i_1<\ldots <i_L<M$.
Since ${\bf x}_{i_0}$ and ${\bf x}_{i_1}$ are linearly independent, we chose $(x,y,z)\in \lZ ^3 \setminus \{ {\bf 0} \} $ such that $\max (|x|,|y|,|z|)\leq 2 p^{2 t_{j+i_1}}$ by Lemma \ref{hyper}.
Since $(s_k)_{k\geq 0}$ is strictly increasing, we have $z \not= 0$.
A straightforward computation shows that
\begin{eqnarray*}
(1-p^{s_{j+i_k}})\alpha = p^{s_{j+i_k}}\frac{x}{z} +\frac{y}{z}-(p^{s_{j+i_k}}\alpha -\alpha -p_{j+i_k}),\\
\alpha -\frac{y}{z}=p^{s_{j+i_k}}\alpha +p^{s_{j+i_k}}\frac{x}{z}-(p^{s_{j+i_k}}\alpha -\alpha -p_{j+i_k})
\end{eqnarray*}
for all $0\leq k\leq L$.
Therefore, we obtain
\begin{eqnarray*}
|\alpha |_p=|(1-p^{s_{j+i_k}})\alpha |_p & \leq & \max \left( p^{-s_{j+i_k}} \left| \frac{x}{z}\right| _p,  \left| \frac{y}{z}\right| _p, p^{-w t_{j+i_k}}\right) \\
& \leq & \max (|z|, p^{-w t_{j+i_k}})\leq 2 p^{2 t_{j+i_1}}.
\end{eqnarray*}
Hence, we have
\begin{eqnarray}\label{ˆê"Ô}
\left| \alpha -\frac{y}{z}\right| _p\leq \max (2 p^{2 t_{j+i_1}-s_{j+i_L}},|z|p^{-s_{j+i_L}},p^{-w t_{j+i_L}})=2 p^{2 t_{j+i_1}-s_{j+i_L}}.
\end{eqnarray}
It follows from Lemma \ref{Liouville inequality} that
\begin{eqnarray}
\left| \alpha -\frac{y}{z}\right| _p & \geq & 2^{-d}(d+1)^{-1}H(\alpha )^{-1} H\left( \frac{y}{z}\right) ^{-d} \nonumber \\
& \geq & 2^{-2 d}(d+1)^{-1}p^{-2 d t_{j+i_1}-s_j} \label{"ñ"Ô} .
\end{eqnarray}
By the properties of $(s_k)_{k\geq 0}$ and $(t_k)_{k\geq 0}$, we have
\begin{eqnarray}\label{ŽO"Ô}
s_{j+i_L}\geq \frac{t_{j+i_L}}{2\kappa +1}=\frac{1}{2\kappa +1}\frac{t_{j+i_L}}{t_{j+i_{L-1}}}\cdots \frac{t_{j+i_2}}{t_{j+i_1}}t_{j+i_1}\geq \frac{2^L t_{j+i_1}}{4\kappa +2}.
\end{eqnarray}
Applying (\ref{ˆê"Ô}), (\ref{"ñ"Ô}), and (\ref{ŽO"Ô}), we obtain
\begin{eqnarray*}
t_{j+i_1}\leq \frac{(4\kappa +2)\log _p (2^{2 d+1}(d+1))}{2^L-(4\kappa +2)(2 d+3)}\leq 2,
\end{eqnarray*}
which is contradiction.

Hence, we get $M\leq LN$.
By the definition of $M$, we have
\begin{eqnarray*}
H(\alpha )^{\chi } & \leq & p^{w c^M t_j} \leq p^{w c^{M+1}t_{j-1}} \\
& \leq & p^{w c^{M+1}(2\kappa +1)s_{j-1}} \leq H(\alpha )^{w c^{M+1}(2\kappa +1)}.
\end{eqnarray*}
Therefore, we obtain
\begin{eqnarray*}
|\xi -\alpha |_p\geq H(\alpha )^{-w c^{L N+1}(2\kappa +1)},
\end{eqnarray*}
which implies
\begin{eqnarray*}
w_d ^{*}(\xi )\leq \max (w_1 (\xi ),w c^{L N+1}(2\kappa +1)).
\end{eqnarray*}
This completes the proof.
\end{proof}

\begin{proof}[Proof of Theorem \ref{main2}]
We first assume that $\xi $ is not a $U_1$-number, that is, $w_1(\xi )$ is finite.
Then $\Dio ({\bf a})$ is finite by Lemma \ref{•Ð'¤}.

We next assume that $\Dio ({\bf a})$ is finite.
For $n\geq n_0$, take finite words $U_n , V_n$ and a rational number $w_n$ satisfying Lemma \ref{'g'ݍ‡'킹} (i)-(vi).
For $n\geq n_0$, we set rational numbers
\begin{eqnarray*}
\alpha _n:=\sum _{i=0}^{\infty }b_i ^{(n)} p^i
\end{eqnarray*}
where $(b_i ^{(n)})_{i\geq 0}$ is the infinite word $U_n \overline{V_n}$.
Since $\xi $ and $\alpha _n$ have the same first $|U_n V_n ^{w_n}|$-th digits, we obtain
\begin{eqnarray*}
\left| \xi -\alpha _n \right| \leq p^{-\left( 1+\frac{1}{4\kappa +2}\right) |U_n V_n|}.
\end{eqnarray*}
Take a real number $\delta $ which is greater than $\Dio ({\bf a})$.
Note that $\delta >1$.
By the definition of the Diophantine exponent, there exists an integer $n_1\geq n_0$ such that for all $n\geq n_1$
\begin{eqnarray*}
\left| \xi -\alpha _n \right| \geq p^{-\delta |U_n V_n|}.
\end{eqnarray*}
We define a positive integer sequence $(n_k)_{k\geq 1}$ by $n_{k+1}=2(\kappa +1)n_k$ for $k\geq 1$.
Set $\beta _k:=p^{|U_{n_k}V_{n_k}|}$ for $k\geq 1$.
It follows from Lemma \ref{'g'ݍ‡'킹} (iii),(vi) that for $n\geq 1$
\begin{eqnarray*}
\beta _k<\beta _{k+1}\leq \beta _k ^{4(\kappa +1)^2}.
\end{eqnarray*}
By Lemma \ref{''³ãŒÀ}, we have $H(\alpha _{n_k})\leq \beta _k$ for $k\geq 1$.
Hence, we obtain (\ref{last}) by Lemma \ref{–³—""x"äŠr}.
\end{proof}

\subsection*{Acknowledgements}
I would like to thank Prof.~Shigeki Akiyama for improving the language and structure of this paper.
I wish to thank the referee for a careful reading and several helpful comments.

\end{document}